\documentclass[reqno,11pt]{article} 
\usepackage{amsmath, amssymb}
\usepackage{amsthm} 
\newcommand{\Mod}[1]{\ (\textup{mod}\ #1)}
\theoremstyle{plain} 
\newtheorem{theorem}{\indent\sc Theorem}[section]
\newtheorem{lemma}[theorem]{\indent\sc Lemma}
\newtheorem{corollary}[theorem]{\indent\sc Corollary}
\newtheorem{proposition}[theorem]{\indent\sc Proposition}

\theoremstyle{definition} 

\newtheorem{remark}[theorem]{\indent\sc Remark}
\newtheorem{example}[theorem]{\indent\sc Example}

\makeatletter
\def\address#1#2{\begingroup
\noindent\parbox[t]{7.8cm}{%
\small{\scshape\ignorespaces#1}\par\vskip1ex
\noindent\small{\itshape E-mail address}%
\/: #2\par\vskip4ex}\hfill%
\endgroup}%
\makeatother
\title{On some Fricke families and application to the Lang-Schertz conjecture} 
\author{
\textsc{Ho Yun Jung, Ja Kyung Koo and Dong Hwa Shin$^*$} 
}
\date{} 
\begin{document}

\allowdisplaybreaks

\maketitle

\footnote{ 
2010 \textit{Mathematics Subject Classification}. Primary 11G15, Secondary 11F03.}
\footnote{ 
\textit{Key words and phrases}. Complex multiplication, Fricke families, modular functions.}
\footnote{
\thanks{
$^*$The corresponding author was supported by Hankuk University of Foreign Studies Research Fund of 2014.} }

\begin{abstract}
We first investigate two kinds of Fricke families consisting of
Fricke functions and Siegel functions, respectively. And, in terms
of their special values we generate ray class fields of imaginary
quadratic fields, which is related to the Lang-Schertz conjecture.
\end{abstract}

\section {Introduction}

For a positive integer $N$, let $\mathcal{F}_N$ be the field of all meromorphic modular functions
of level $N$ whose Fourier coefficients lie in the $N$-th cyclotomic field
$\mathbb{Q}(\zeta_N)$ with $\zeta_N=e^{2\pi i/N}$.
Then it is well-known that $\mathcal{F}_1$ is generated over $\mathbb{Q}$ by the elliptic modular function
\begin{equation*}
j(\tau)=1/q+744+196884q+21493760q^2
+\cdots\quad(\tau\in\mathbb{H},~\textrm{the complex upper half plane}),
\end{equation*}
where $q=e^{2\pi i\tau}$.
Furthermore, $\mathcal{F}_N$ is a Galois extension of $\mathcal{F}_1$ whose Galois group is isomorphic to $\mathrm{GL}_2(\mathbb{Z}/N\mathbb{Z})/\{\pm I_2\}$
(\cite[$\S$6.1--6.2]{Shimura}).
\par
For $N\geq2$ we let
\begin{equation*}
  \mathcal{V}_N=\{\mathbf{v}\in\mathbb{Q}^2~|~
  \textrm{$\mathbf{v}$ has primitive denominator $N$}\}
\end{equation*}
, that is, $\mathbf{v}\in\mathbb{Q}^2$ belongs to $\mathcal{V}_N$ if
and only if $N$ is the smallest positive integer satisfying
$N\mathbf{v} \in\mathbb{Z}^2$. We call a family
$\{h_\mathbf{v}(\tau)\}_{\mathbf{v}\in\mathcal{V}_N}$ of functions
in $\mathcal{F}_N$
 a \textit{Fricke family of level $N$} if
\begin{itemize}
\item[(F1)] $h_\mathbf{v}(\tau)$ is weakly holomorphic (namely, $h_\mathbf{v}(\tau)$ is
holomorphic on $\mathbb{H}$),
\item[(F2)] $h_\mathbf{v}(\tau)$ depends only on $\pm\mathbf{v}\Mod{\mathbb{Z}^2}$,
\item[(F3)] $h_\mathbf{v}(\tau)^\gamma=h_{{^t}\gamma\mathbf{v}}(\tau)$
for all $\gamma\in\mathrm{GL}_2(\mathbb{Z}/N\mathbb{Z})/\{\pm I_2\}\simeq
\mathrm{Gal}(\mathcal{F}_N/\mathcal{F}_1)$, where ${^t}\gamma$ indicates the transpose of $\gamma$.
\end{itemize}
In this paper, we shall deal with two kinds of Fricke families
$\{f_\mathbf{v}(\tau)\}_{\mathbf{v}\in\mathcal{V}_N}$ and
$\{g_\mathbf{v}(\tau)^{12N}\}_{\mathbf{v}\in\mathcal{V}_N}$; one consisting of Fricke functions and
the other consisting of $12N$-th powers of Siegel functions ($\S$\ref{sectwo}).
\par
Let $K$ be an imaginary quadratic field of discriminant $d_K$ other than $\mathbb{Q}(\sqrt{-1})$ and $\mathbb{Q}(\sqrt{-3})$, and $\mathfrak{n}$ be a
proper nontrivial ideal of the ring of integers $\mathcal{O}_K$ of $K$.
Furthermore, let $N$ ($\geq2$) be the smallest positive integer in $\mathfrak{n}$ and
$C$ be a ray class in the ray class group $\mathrm{Cl}(\mathfrak{n})$ of $K$ modulo $\mathfrak{n}$.  For a Fricke family $\{h_\mathbf{v}(\tau)\}_{\mathbf{v}\in\mathcal{V}_N}$ of level $N$, we shall define in
$\S$\ref{invariantdef}
the \textit{Fricke invariant} $h_\mathfrak{n}(C)$ which depends only on $\mathfrak{n}$ and $C$.
Then the first main theorem of this paper asserts that $f_\mathfrak{n}(C)$ generates
the ray class field $K_\mathfrak{n}$ of $K$ modulo $\mathfrak{n}$
over the Hilbert class field $H_K$ of $K$
(Theorem \ref{Frickesingular}).
\par
On the other hand, Lang (\cite[p.292]{Lang}) and Schertz (\cite{Schertz}) conjectured that
$g_\mathfrak{n}^{12N}(C)$, which is called the \textit{Siegel-Ramachandra invariant}
modulo $\mathfrak{n}$ at $C$, generates $K_\mathfrak{n}$ over $H_K$ (even, over $K$). Recently, Cho (\cite{Cho}) gave a conditional proof by adopting Schertz's idea and using the second Kronecker limit formula as follows:
Let $\mathfrak{n}=\prod_{k=1}^r\mathfrak{p}_k^{e_k}$ be the prime ideal factorization of $\mathfrak{n}$. If the exponent of the quotient group
$(\mathcal{O}_K/\mathfrak{p}_k^{e_k})^\times/\{\alpha+\mathfrak{p}_k^{e_k}~|~
\alpha\in\mathcal{O}_K^\times\}$ is greater than $2$ for every $k=1,\ldots,r$, then $g_\mathfrak{n}^{12N}(C)$ generates $K_\mathfrak{n}$ over $K$.
\par
As the second main theorem we shall present a new conditional proof of the Lang-Schertz conjecture
(Theorem \ref{Siegelinvariant} and Corollary \ref{maincor}). We shall further show that
the $6N$-th root, which will give a relatively small power, of a certain quotient of Siegel-Ramachandra invariants
generates $K_\mathfrak{n}$ over $H_K$ when $d_K\equiv N\equiv0\Mod{4}$,
$|d_K|\geq4N^{4/3}$ and $\mathfrak{n}=N\mathcal{O}_K$ (Theorem \ref{smallinvariant}).
To this end, we shall find
some relations between Fricke and Siegel functions (Lemma \ref{fg}), and
 make use of an explicit version of Shimura's reciprocity law due to Stevenhagen (Proposition  \ref{Shimura}).
And, we note that these invariants have minimal polynomials with
(relatively) small coefficients (Example \ref{lastexample}).

\section {Fricke families}\label{sectwo}

Let $\Lambda$ be a lattice in $\mathbb{C}$. The \textit{Weierstrass $\wp$-function} relative to $\Lambda$ is defined by
\begin{equation*}
\wp(z;\Lambda)=1/z^2+\sum_{\omega\in\Lambda\setminus\{0\}}
(1/(z-\omega)^2-1/\omega^2)\quad(z\in\mathbb{C}).
\end{equation*}
Then it is a meromorphic function on $z$ and is periodic with
respect to $\Lambda$.

\begin{lemma}\label{pval}
If $z_1,z_2\in\mathbb{C}\setminus\Lambda$, then
$\wp(z_1;\Lambda)=\wp(z_2;\Lambda)$ if and only if $z_1\equiv\pm z_2\Mod{\Lambda}$.
\end{lemma}
\begin{proof}
See \cite[Chaper IV, $\S$3]{Silverman}.
\end{proof}

Let $N$ ($\geq2$) be an integer and $\mathbf{v}=\left[\begin{matrix}v_1\\v_2\end{matrix}\right]\in\mathcal{V}_N$. We define
\begin{equation}\label{pv}
\wp_{\mathbf{v}}(\tau)=
\wp(v_1\tau+v_2;[\tau,1])
\quad(\tau\in\mathbb{H}),
\end{equation}
which is a weakly holomorphic modular form of level $N$ and weight $2$
(\cite[Chapter 6]{Lang}). We further define auxiliary functions $g_2(\tau)$, $g_3(\tau)$
and $\Delta(\tau)$ on $\mathbb{H}$ by
\begin{eqnarray*}
g_2(\tau)=60\sum_{\omega\in[\tau,1]\setminus\{0\}}1/\omega^4,\quad
g_3(\tau)=140\sum_{\omega\in[\tau,1]\setminus\{0\}}1/\omega^6\quad\textrm{and}\quad
\Delta(\tau)=g_2(\tau)^3-27g_3(\tau)^2,
\end{eqnarray*}
which are holomorphic modular forms of level $1$ and weight $4$, $6$ and $12$, respectively
(\cite[Chapter 3, $\S$2]{Lang}). Now, we define the \textit{Fricke function}
(or, the \textit{first Weber function}) $f_{\mathbf{v}}(\tau)$  by
\begin{equation}\label{Fricke}
f_{\mathbf{v}}(\tau)=-2^73^5(g_2(\tau)g_3(\tau)/\Delta(\tau))
\wp_{\mathbf{v}}(\tau)\quad(\tau\in\mathbb{H}).
\end{equation}

\begin{proposition}\label{Frickefamily}
The family
$\{f_\mathbf{v}(\tau)\}_{\mathbf{v}\in\mathcal{V}_N}$
is a Fricke family of level $N$.
\end{proposition}
\begin{proof}
See \cite[Chapter 6, $\S$2--3]{Lang}.
\end{proof}

The \textit{Weierstrass $\sigma$-function} relative to $\Lambda$ is defined by
\begin{equation*}
\sigma(z;\Lambda)=z\prod_{\omega\in \Lambda\setminus\{0\}}(1-z/\omega)
e^{z/\omega+(1/2)(z/\omega)^2} \quad(z\in\mathbb{C}).
\end{equation*}
Taking logarithmic derivative we come up with the \textit{Weierstrass
$\zeta$-function} as
\begin{equation*}
\zeta(z;\Lambda)=\sigma'(z;\Lambda)/\sigma(z;\Lambda)
=1/z+\sum_{\omega\in
\Lambda\setminus\{0\}}(1/(z-\omega)+1/\omega+
z/\omega^2).
\end{equation*}
Since $\zeta'(z;\Lambda)=-\wp(z;\Lambda)$ is periodic with respect to $\Lambda$, for any $\omega\in\Lambda$ there is
a constant $\eta(\omega;\Lambda)$ so that
\begin{equation*}
\zeta(z+\omega;\Lambda)-\zeta(z;\Lambda)=\eta(\omega;\Lambda).
\end{equation*}
Next, we define
the \textit{Siegel function} $g_\mathbf{v}(\tau)$ by
\begin{equation*}
g_\mathbf{v}(\tau)=
e^{-(1/2)(v_1\eta(\tau;[\tau,1])+
v_2\eta(1;[\tau,1]))
(v_1\tau+v_2)}\sigma(v_1\tau+v_2;[\tau,1])\eta(\tau)^2
\quad(\tau\in\mathbb{H}),
\end{equation*}
where
\begin{equation*}
\eta(\tau)=\sqrt{2\pi}\zeta_8q^{1/24}\prod_{n=1}^\infty
(1-q^n)\quad(q=e^{2\pi i\tau},\tau\in\mathbb{H})
\end{equation*}
is the \textit{Dedekind $\eta$-function}.
By the product formula of the
Weierstrass $\sigma$-function
we get the
$q$-product expression
\begin{equation}\label{FourierSiegel}
g_{\mathbf{v}}(\tau)
=-e^{\pi i v_2(v_1-1)}
q^{(1/2)\mathbf{B}_2(v_1)}
(1-q^{v_1}e^{2\pi iv_2})
\prod_{n=1}^\infty (1-q^{n+v_1}e^{2\pi iv_2})
(1-q^{n-v_1}e^{-2\pi iv_2}),
\end{equation}
where $\mathbf{B}_2(X)=X^2-X+1/6$
is the second Bernoulli polynomial
(\cite[Chapter 18, Theorem 4 and Chapter 19, $\S$2]{Lang}).
Furthermore, we have the $q$-order formula
\begin{equation}\label{OrderSiegel}
\mathrm{ord}_{q}(g_\mathbf{v}(\tau))=
(1/2)\mathbf{B}_2(\langle
v_1\rangle),
\end{equation}
where $\langle X\rangle$ is the fractional part of $X\in\mathbb{R}$
such that $0\leq\langle X\rangle<1$ (\cite[Chapter 2, $\S$1]{K-L}).

\begin{lemma}\label{modularity}
Let $\{m(\mathbf{v})\}_{\mathbf{v}=\left[\begin{smallmatrix}v_1\\v_2\end{smallmatrix}\right]\in\mathcal{V}_N}$ be a family of
integers such that $m(\mathbf{v})=0$ except for finitely many $\mathbf{v}$.
If the family satisfies
\begin{eqnarray*}
&&\sum_{\mathbf{v}} m(\mathbf{v})(Nv_1)^2\equiv
\sum_{\mathbf{v}}
m(\mathbf{v})(Nv_2)^2\equiv0\Mod{\gcd(2,N)\cdot N},\\
&&\sum_{\mathbf{v}}
m(\mathbf{v})(Nv_1)(Nv_2)\equiv0\Mod{N},\\
&&\gcd(12,N)\cdot\sum_{\mathbf{v}} m(\mathbf{v})\equiv0\Mod{12},
\end{eqnarray*}
then $\zeta\prod_\mathbf{v}
g_{\mathbf{v}}(\tau)^{m(\mathbf{v})}$
belongs to $\mathcal{F}_N$, where $\zeta=\prod_\mathbf{v}e^{\pi iv_2(1-v_1)m(\mathbf{v})}\in\mathbb{Q}(\zeta_{2N^2})$.
\end{lemma}
\begin{proof}
See \cite[Chapter 3, Theorems 5.2 and 5.3]{K-L} and (\ref{FourierSiegel}).
\end{proof}

\begin{proposition}\label{Siegelfamily}
The family
$\{g_\mathbf{v}(\tau)^{12N}\}_{\mathbf{v}\in\mathcal{V}_N}$
is a Fricke family of level $N$.
\end{proposition}
\begin{proof}
See \cite[Chapter 2, Proposition 1.3]{K-L}.
\end{proof}

\begin{lemma}\label{fg}
We further obtain the following results on modular functions.
\begin{itemize}
\item[\textup{(i)}]
If $h(\tau)$ is a weakly holomorphic function in $\mathcal{F}_1$, then it
is a polynomial in $j(\tau)$ over $\mathbb{Q}$.
\item[\textup{(ii)}] We have
$\mathcal{F}_1(
f_{\left[\begin{smallmatrix}0\\1/N\end{smallmatrix}\right]}(\tau))=\mathcal{F}_1(
g_{\left[\begin{smallmatrix}0\\1/N\end{smallmatrix}\right]}(\tau)^{12N})$.
\item[\textup{(iii)}] We get the relation
\begin{equation*}
f_{\left[\begin{smallmatrix}0\\1/N\end{smallmatrix}\right]}(\tau)=
\frac{p(j(\tau),g_{\left[\begin{smallmatrix}0\\1/N\end{smallmatrix}\right]}(\tau)^{12N})}
{\mathrm{disc}(
g_{\left[\begin{smallmatrix}0\\1/N\end{smallmatrix}\right]}(\tau)^{12N},\mathcal{F}_1)}
\quad\textrm{for some polynomial $p(X,Y)\in\mathbb{Q}[X,Y]$}.
\end{equation*}
\end{itemize}
\end{lemma}
\begin{proof}
(i) See \cite[Chapter 5, Theorem 2]{Lang}.\\
(ii) Let
$L=\mathcal{F}_1(f_{\left[\begin{smallmatrix}0\\1/N\end{smallmatrix}\right]}(\tau))$ and
$R=\mathcal{F}_1(g_{\left[\begin{smallmatrix}0\\1/N\end{smallmatrix}\right]}(\tau)^{12N})$,
which are intermediate subfields of the extension $\mathcal{F}_N/\mathcal{F}_1$
by Propositions \ref{Frickefamily} and \ref{Siegelfamily}.
Let $\gamma=\left[\begin{matrix}x&y\\z&w\end{matrix}\right]\in\mathrm{GL}_2(\mathbb{Z}/N\mathbb{Z})/
\{\pm I_2\}\simeq\mathrm{Gal}(\mathcal{F}_N/\mathcal{F}_1)$. Then we deduce that
\begin{eqnarray*}
\gamma\in\mathrm{Gal}(\mathcal{F}_N/L)&\Longleftrightarrow&
f_{\left[\begin{smallmatrix}0\\1/N\end{smallmatrix}\right]}(\tau)^\gamma=
f_{\left[\begin{smallmatrix}0\\1/N\end{smallmatrix}\right]}(\tau)\\
&\Longleftrightarrow&
f_{{^t}\gamma\left[\begin{smallmatrix}0\\1/N\end{smallmatrix}\right]}(\tau)=
f_{\left[\begin{smallmatrix}0\\1/N\end{smallmatrix}\right]}(\tau)\quad\textrm{by Proposition \ref{Frickefamily}, (F2) and (F3)}\\
&\Longleftrightarrow&
f_{\left[\begin{smallmatrix}z/N\\w/N\end{smallmatrix}\right]}(\tau)=
f_{\left[\begin{smallmatrix}0\\1/N\end{smallmatrix}\right]}(\tau)\\
&\Longleftrightarrow&
\wp((z/N)\tau+w/N;[\tau,1])=\wp(1/N;[\tau,1])\quad\textrm{by the definitions (\ref{pv}) and (\ref{Fricke})}\\
&\Longleftrightarrow&
(z/N)\tau+w/N\equiv\pm1/N\Mod{[\tau,1]}\quad\textrm{by Lemma \ref{pval}}\\
&\Longleftrightarrow&z\equiv0,~
w\equiv\pm1\Mod{N}.
\end{eqnarray*}
Thus we obtain
\begin{equation}\label{GalL}
\mathrm{Gal}(\mathcal{F}_N/L)=\{\gamma\in\mathrm{GL}_2(\mathbb{Z}/N\mathbb{Z})~|~
\gamma\equiv\pm\left[\begin{matrix}* & * \\ 0 & 1\end{matrix}\right]\Mod{N}\}/\{\pm I_2\},
\end{equation}
and it follows from Proposition \ref{Siegelfamily}, (F2) and (F3) that every element of $\mathrm{Gal}(\mathcal{F}_N/L)$
leaves $g_{\left[\begin{smallmatrix}0\\1/N\end{smallmatrix}\right]}(\tau)^{12N}$ fixed.
This implies that
$\mathrm{Gal}(\mathcal{F}_N/L)\subseteq\mathrm{Gal}(\mathcal{F}_N/R)$.
\par
Conversely, let
$\rho=\left[\begin{matrix}a&b\\c&d\end{matrix}\right]\in\mathrm{Gal}(\mathcal{F}_N/R)$.
We then derive by Proposition \ref{Siegelfamily}, (F2) and (F3) that
\begin{equation}\label{eq1}
g_{\left[\begin{smallmatrix}c/N\\d/N\end{smallmatrix}\right]}(\tau)^{12N}
=g_{\left[\begin{smallmatrix}0\\1/N\end{smallmatrix}\right]}(\tau)^{12N}.
\end{equation}
The action of $\left[\begin{matrix}0 &-1\\1&0\end{matrix}\right]$ on
both sides of (\ref{eq1}) yields
\begin{equation}\label{eq2}
g_{\left[\begin{smallmatrix}d/N\\-c/N\end{smallmatrix}\right]}(\tau)^{12N}
=g_{\left[\begin{smallmatrix}1/N\\0\end{smallmatrix}\right]}(\tau)^{12N}.
\end{equation}
By applying the $q$-order formula (\ref{OrderSiegel}) to the expressions (\ref{eq1}) and (\ref{eq2})
we attain
\begin{equation*}
6N\mathbf{B}_2(\langle c/N\rangle)=6N\mathbf{B}_2(0)\quad\textrm{and}\quad
6N\mathbf{B}_2(\langle d/N\rangle)=6N\mathbf{B}_2(1/N).
\end{equation*}
Now, we deduce
by the shape of the graph $Y=\mathbf{B}_2(X)$
that
\begin{equation*}
c\equiv0,~d\equiv\pm1\Mod{N}.
\end{equation*}
This, together with (\ref{GalL}), shows that $\mathrm{Gal}(\mathcal{F}_N/L)\supseteq\mathrm{Gal}(\mathcal{F}_N/R)$. Thus we achieve $\mathrm{Gal}(\mathcal{F}_N/L)=\mathrm{Gal}(\mathcal{F}_N/R)$, and hence
$L=R$, as desired. \\
(iii) For simplicity, let $f=f_{\left[\begin{smallmatrix}0\\1/N\end{smallmatrix}\right]}(\tau)$
and $g=g_{\left[\begin{smallmatrix}0\\1/N\end{smallmatrix}\right]}(\tau)^{12N}$.
By (i) we can express $f$ as
\begin{equation*}
f=c_0+c_1g+\cdots+c_{\ell-1}g^{\ell-1}\quad\textrm{for some}~c_0,c_1,\ldots,c_{\ell-1}\in
\mathcal{F}_1,
\end{equation*}
where $\ell=[\mathcal{F}_1(g):\mathcal{F}_1]$.
Multiplying
both sides by $g^k$ ($k=0,1,\ldots,\ell-1$)
and taking traces $\mathrm{Tr}=\mathrm{Tr}_{\mathcal{F}_1(g)/\mathcal{F}_1}$
yields
\begin{equation*}
\mathrm{Tr}(fg^k)=c_0\mathrm{Tr}(g^k)+c_1\mathrm{Tr}(g^{k+1})
+\cdots
+c_{\ell-1}\mathrm{Tr}(g^{k+\ell-1}).
\end{equation*}
So we obtain a linear system (in unknowns $c_0,c_1,c_2,\ldots, c_{\ell-1}$)
\begin{equation*}
T\left[\begin{matrix}c_0\\c_1\\
\vdots\\c_{\ell-1}
\end{matrix}\right]=
\left[\begin{matrix}\mathrm{Tr}(f)\\
\mathrm{Tr}(fg)\\
\vdots\\
\mathrm{Tr}(fg^{\ell-1})\end{matrix}\right],~\textrm{where}~T=
\left[\begin{matrix}
\mathrm{Tr}(1) & \mathrm{Tr}(g) & \cdots & \mathrm{Tr}(g^{\ell-1})\\
\mathrm{Tr}(g) & \mathrm{Tr}(g^2) & \cdots & \mathrm{Tr}(g^\ell)\\
\vdots & \vdots &\ddots& \vdots\\
\mathrm{Tr}(g^{\ell-1}) & \mathrm{Tr}(g^{\ell}) & \cdots
& \mathrm{Tr}(g^{2\ell-2})
\end{matrix}\right].
\end{equation*}
Since $f$ and $g$ are weakly holomorphic by (F1), so are all entries of the augmented matrix of the
above linear system. Thus we get by (i) that
\begin{equation*}
c_0,c_1,\ldots, c_{\ell-1}\in(1/\det(T))\mathbb{Q}[j].
\end{equation*}
On the other hand, let $g_1,g_2,\ldots,g_\ell$ be all the zeros of
$\min(g,\mathcal{F}_1)$. Then we see that
\begin{eqnarray*}
\det(T)&=&\left|\begin{matrix}\sum_{k=1}^\ell g_k^0
& \sum_{k=1}^\ell g_k^1 & \cdots & \sum_{k=1}^\ell g_k^{\ell-1}\\
\sum_{k=1}^\ell g_k^1
& \sum_{k=1}^\ell g_k^2 & \cdots & \sum_{k=1}^\ell g_k^{\ell}\\
\vdots
& \vdots & \ddots & \vdots\\
\sum_{k=1}^\ell g_k^{\ell-1}
& \sum_{k=1}^\ell g_k^\ell & \cdots & \sum_{k=1}^\ell g_k^{2\ell-2}\\
\end{matrix}\right|\\
&=&\left|\begin{matrix}
g_1^0 & g_2^0 & \cdots & g_\ell^0\\
g_1^1 & g_2^1 & \cdots & g_\ell^1\\
\vdots & \vdots & \ddots & \vdots\\
g_1^{\ell-1} & g_2^{\ell-1} & \cdots & g_\ell^{\ell-1}
\end{matrix}\right|
\cdot
\left|\begin{matrix}
g_1^0 & g_1^1 & \cdots & g_1^{\ell-1}\\
g_2^0 & g_2^1 & \cdots & g_2^{\ell-1}\\
\vdots& \vdots & \ddots & \vdots\\
g_\ell^0 & g_\ell^1 & \cdots & g_\ell ^{\ell-1}
\end{matrix}\right|\\
&=&\prod_{1\leq k_1<k_2\leq \ell}(g_{k_1}-g_{k_2})^2\quad\textrm{by the Vandermonde determinant formula}\\
&=&\mathrm{disc}(g,\mathcal{F}_1).
\end{eqnarray*}
This proves (iii).
\end{proof}

\begin{remark}\label{zeros}
Define an equivalence relation $\sim$ on $\mathcal{V}_N$ as follows:
\begin{equation*}
\mathbf{u}\sim\mathbf{v}\quad\textrm{if and only if}\quad
\mathbf{u}\equiv\pm\mathbf{v}\Mod{\mathbb{Z}^2}.
\end{equation*}
Then, in a similar way as in the proof of Lemma \ref{fg}(ii), one can readily show that
\begin{equation*}
f_\mathbf{v}(\tau)~\textrm{and}~
g_\mathbf{v}(\tau)^{12N}\quad\textrm{for}~\mathbf{v}\in\mathcal{V}_N/\sim
\end{equation*}
represent all the distinct zeros of
$\min(f_{\left[\begin{smallmatrix}0\\1/N\end{smallmatrix}\right]}(\tau),\mathcal{F}_1)$
and
$\min(g_{\left[\begin{smallmatrix}0\\1/N\end{smallmatrix}\right]}(\tau)^{12N},\mathcal{F}_1)$, respectively.
\end{remark}

\section {Generation of class fields}\label{invariantdef}

Let $K$ be an imaginary quadratic field and $\mathcal{O}_K$ be its ring of integers.
Let $\mathfrak{n}$ be a proper nontrivial ideal of $\mathcal{O}_K$, $N$ ($\geq2$) be the
smallest positive integer in $\mathfrak{n}$ and $C$ be a ray class
in the ray class group $\mathrm{Cl}(\mathfrak{n})$ of $K$ modulo $\mathfrak{n}$.
We take an integral ideal $\mathfrak{c}$ in the class $C$ and let
\begin{eqnarray*}
\mathfrak{n}\mathfrak{c}^{-1}&=&[\omega_1,\omega_2]\quad\textrm{for some}~ \omega_1,\omega_2\in\mathbb{C}~\textrm{with}~\omega=\omega_1/\omega_2\in\mathbb{H},\\
1&=&(a/N)\omega_1+(b/N)\omega_2\quad\textrm{for some}~a,b\in\mathbb{Z}.
\end{eqnarray*}
For a given Fricke family $\{h_\mathbf{v}(\tau)\}_{\mathbf{v}\in\mathcal{V}_N}$ of level $N$, we define the \textit{Fricke invariant} $h_\mathfrak{n}(C)$ modulo $\mathfrak{n}$ at $C$ by
\begin{equation}\label{definvariant}
h_\mathfrak{n}(C)=h_{\left[\begin{smallmatrix}a/N\\b/N\end{smallmatrix}\right]}
(\omega).
\end{equation}
This value depends only on $\mathfrak{n}$ and $C$, not on the choices of $\mathfrak{c}$, $\omega_1$ and $\omega_2$ (\cite[Chapter 11, $\S$1]{K-L}).

\begin{proposition}\label{invariant}
The Fricke invariant $h_\mathfrak{n}(C)$ lies in the ray class field $K_\mathfrak{n}$
of $K$ modulo $\mathfrak{n}$ and satisfies the following transformation formula:
\begin{equation*}
h_\mathfrak{n}(C)^{\sigma_\mathfrak{n}(C')}=h_\mathfrak{n}(CC')\quad\textrm{for any class}~C'\in\mathrm{Cl}(\mathfrak{n}),
\end{equation*}
where $\sigma_\mathfrak{n}:\mathrm{Cl}(\mathfrak{n})\rightarrow\mathrm{Gal}(K_\mathfrak{n}/K)$ is the Artin reciprocity map.
Furthermore, the algebraic number $g_\mathfrak{n}^{12N}(C)/g_\mathfrak{n}^{12N}(C')$ is a unit.
\end{proposition}
\begin{proof}
See \cite[Chapter 11, Theorems 1.1 and 1.2]{K-L}.
\end{proof}

\begin{theorem}\label{Frickesingular}
Assume that $K$ is different from $\mathbb{Q}(\sqrt{-1})$ and $\mathbb{Q}(\sqrt{-3})$.
Then the first Fricke invariant $f_\mathfrak{n}(C)$ generates $K_\mathfrak{n}$ over the Hilbert class field $H_K$ of $K$.
\end{theorem}
\begin{proof}
Let $C_0$ be the identity class of $\mathrm{Cl}(\mathfrak{n})$.
Since $K_\mathfrak{n}$ is a finite abelian extension of $K$,
it suffices to show that $f_\mathfrak{n}(C_0)$ generates $K_\mathfrak{n}$ over $H_K$.
Let
\begin{eqnarray*}
I_K(\mathfrak{n})&=&\textrm{the group of fractional ideals of $K$
prime to $\mathfrak{n}$},\\
P_K(\mathfrak{n})&=&\langle\alpha\mathcal{O}_K~|~\alpha\in\mathcal{O}_K~
\textrm{such that $\alpha\mathcal{O}_K$ is prime to $\mathfrak{n}$}\rangle
\quad(\subseteq I_K(\mathfrak{n})),\\
P_{K,1}(\mathfrak{n})&=&\langle\alpha\mathcal{O}_K~|~\alpha\in\mathcal{O}_K~\textrm{such that}~
\alpha\equiv1\Mod{\mathfrak{n}}\rangle
\quad(\subseteq P_K(\mathfrak{n})).
\end{eqnarray*}
Since $\mathrm{Gal}(K_\mathfrak{n}/K)\simeq I_K(\mathfrak{n})/P_{K,1}(\mathfrak{n})$ and
$\mathrm{Gal}(H_K/K)\simeq I_K(\mathfrak{n})/P_K(\mathfrak{n})$ (\cite[Chapters IV and V]{Janusz}), we get
\begin{equation}\label{GalKmHK}
\mathrm{Gal}(K_\mathfrak{n}/H_K)\simeq P_K(\mathfrak{n})/P_{K,1}(\mathfrak{n}).
\end{equation}
\par
Assume that a class $D$ in $P_K(\mathfrak{n})/P_{K,1}(\mathfrak{n})$
leaves $f_\mathfrak{n}(C_0)$ fixed via the Artin reciprocity law.
Here, we may assume that $D=[\alpha\mathcal{O}_K]$ for
some $\alpha\in\mathcal{O}_K$ such that $\alpha\mathcal{O}_K$ is
prime to $\mathfrak{n}$,
since $P_K(\mathfrak{n})/P_{K,1}(\mathfrak{n})$ is a finite group.
Take $\mathfrak{c}=\mathcal{O}_K\in C_0$ and let
\begin{eqnarray}
\mathfrak{n}\mathfrak{c}^{-1}&=&\mathfrak{n}~=~[\omega_1,\omega_2]\quad\textrm{for some}~\omega_1,\omega_2\in\mathbb{C}~\textrm{with}~\omega=\omega_1/\omega_2\in\mathbb{H},\label{m1-1}\\
1&=&(a/N)\omega_1+(b/N)\omega_2\quad\textrm{for some}~a,b\in\mathbb{Z}.\label{cd}
\end{eqnarray}
We then have
\begin{eqnarray}
\mathfrak{n}(\alpha\mathcal{O}_K)^{-1}&=&[\omega_1\alpha^{-1},\omega_2\alpha^{-1}],\label{malpha-1}\\
1&=&(r/N)(\omega_1\alpha^{-1})
+(s/N)(\omega_2\alpha^{-1})\quad\textrm{for some}~r,s\in\mathbb{Z}.\label{rs}
\end{eqnarray}
Now we attain that
\begin{eqnarray*}
f_\mathfrak{n}(C_0)&=&f_{\left[\begin{smallmatrix}a/N\\b/N\end{smallmatrix}\right]}
(\omega)\quad\textrm{by (\ref{m1-1}), (\ref{cd}) and the definition (\ref{definvariant})}\\
&=&f_\mathfrak{n}(C_0)^{\sigma_\mathfrak{n}(D)}\\
&=&f_\mathfrak{n}(D)\quad\textrm{by Proposition \ref{invariant} and the fact that $C_0$
is the identity class of $\mathrm{Cl}(\mathfrak{n})$}\\
&=&f_{\left[\begin{smallmatrix}r/N\\s/N\end{smallmatrix}\right]}
(\omega_1\alpha^{-1}/\omega_2\alpha^{-1})
\quad\textrm{by (\ref{malpha-1}), (\ref{rs}) and the definition (\ref{definvariant})}\\
&=&f_{\left[\begin{smallmatrix}r/N\\s/N\end{smallmatrix}\right]}(\omega).
\end{eqnarray*}
Note that since $K$ is different from $\mathbb{Q}(\sqrt{-1})$ and $\mathbb{Q}(\sqrt{-3})$,
we get $g_2(\omega),g_3(\omega)\neq0$ (\cite[Chapter 3, Theorem3]{Lang}).
Thus we obtain by the definition (\ref{Fricke}) and Lemma \ref{pval} that
\begin{equation*}
(a/N)\omega+b/N\equiv\pm((r/N)\omega+s/N)
\Mod{[\omega,1]}.
\end{equation*}
It then follows that
\begin{equation*}
(a/N)\omega_1+(b/N)\omega_2\equiv
\pm((r/N)\omega_1+(s/N)\omega_2)\Mod{[\omega_1,\omega_2]},
\end{equation*}
and hence
\begin{equation*}
1\equiv\pm\alpha\Mod{\mathfrak{n}}
\end{equation*}
by (\ref{cd}), (\ref{rs}) and (\ref{m1-1}). This shows that
the class $D=[\alpha\mathcal{O}_K]$ gives rise to the identity
of $\mathrm{Gal}(K_\mathfrak{n}/H_K)$ via the Artin reciprocity map
by (\ref{GalKmHK}). Therefore, we conclude by Galois theory that $f_\mathfrak{n}(C_0)$ generates  $K_\mathfrak{n}$ over $H_K$.
\end{proof}

\section {Siegel-Ramachandra invariants}

Let $K$ be an imaginary quadratic field other than $\mathbb{Q}(\sqrt{-1})$
and $\mathbb{Q}(\sqrt{-3})$.
Let $d_K$ be its discriminant and set
\begin{equation*}
\tau_K=\left\{\begin{array}{ll}
(-1+\sqrt{d}_K)/2 & \textrm{if}~d_K\equiv1\Mod{4},\\
\sqrt{d_K}/2 & \textrm{if}~d_K\equiv0\Mod{4}
\end{array}\right.
\end{equation*}
so that $\tau_K\in\mathbb{H}$ and $\mathcal{O}_K=[\tau_K,1]$.
Then, as is well known, the special value $j(\tau_K)$ generates $H_K$ over $K$ (\cite[Chapter 10, Theorem 1]{Lang}).
Let $\mathfrak{n}$ be a proper nontrivial ideal of $\mathcal{O}_K$, $N$ ($\geq2$) be the smallest positive integer in $\mathfrak{n}$ and $C\in\mathrm{Cl}(\mathfrak{n})$.
We call the Fricke invariant $g_\mathfrak{n}^{12N}(C)$ the
\textit{Siegel-Ramachandra invariant} modulo $\mathfrak{n}$ at $C$ (\cite{Ramachandra}).
We further let
\begin{equation*}
d_N(\tau)=\mathrm{disc}(g_{\left[\begin{smallmatrix}0\\1/N\end{smallmatrix}\right]}(\tau)^{12N},\mathcal{F}_1).
\end{equation*}

\begin{theorem}\label{Siegelinvariant}
If the special value $d_N(\tau_K)$ is nonzero, then $g_\mathfrak{n}^{12N}(C)$ generates $K_\mathfrak{n}$ over $H_K$.
\end{theorem}
\begin{proof}
As in the proof of Theorem \ref{Frickesingular} we let $C=C_0$ (the identity class of $\mathrm{Cl}(\mathfrak{n})$) and
\begin{equation*}
f_\mathfrak{n}(C_0)=f_{\left[\begin{smallmatrix}a/N\\b/N\end{smallmatrix}\right]}(\omega)
\quad\textrm{and}\quad
g_\mathfrak{n}^{12N}(C_0)=g_{\left[\begin{smallmatrix}a/N\\b/N\end{smallmatrix}\right]}(\omega)^{12N}\quad
\textrm{for some}~
\left[\begin{matrix}a/N\\b/N\end{matrix}\right]\in\mathcal{V}_N~\textrm{and}~\omega\in\mathbb{H}.
\end{equation*}
\par
Since $d_N(\tau)$
is weakly holomorphic by Remark \ref{zeros},
Proposition \ref{Siegelfamily} and (F1),
we get by Lemma \ref{fg}(i) that
\begin{equation}\label{d_Nd}
d_N(\tau)=d(j(\tau))\quad\textrm{for
some polynomial}~d(X)\in\mathbb{Q}[X].
\end{equation}
Then we have $d_N(\tau_K)=d(j(\tau_K))\neq0$ by assumption, and hence
$d(j(\omega))\neq0$ because $j(\omega)$ is a Galois conjugate of
$j(\tau_K)$ over $K$ (\cite[Chapter 10, Theorem 1]{Lang}).
\par
Now, take an element $\gamma\in\mathrm{GL}_2(\mathbb{Z}/N\mathbb{Z})/\{\pm I_2\}\simeq
\mathrm{Gal}(\mathcal{F}_N/\mathcal{F}_1)$ such that
$\gamma\equiv\pm\left[\begin{matrix}* & *\\a&b\end{matrix}\right]\Mod{N}$.
We then derive that
\begin{eqnarray*}
f_\mathfrak{n}(C_0)&=&
f_{\left[\begin{smallmatrix}a/N\\b/N\end{smallmatrix}\right]}(\omega)\\
&=&f_{{^t}\gamma\left[\begin{smallmatrix}0\\1/N\end{smallmatrix}\right]}(\omega)
\quad\textrm{by Proposition \ref{Frickefamily} and (F2)}\\
&=&(f_{\left[\begin{smallmatrix}0\\1/N\end{smallmatrix}\right]}(\tau))^\gamma(\omega)
\quad\textrm{by (F3)}\\
&=&(p(j(\tau),g_{\left[\begin{smallmatrix}0\\1/N\end{smallmatrix}\right]}(\tau)^{12N})
/d(j(\tau)))^\gamma(\omega)\\
&&\textrm{for some polynomial}~p(X,Y)\in\mathbb{Q}[X,Y]~\textrm{by Lemma \ref{fg}(iii)}\\
&=&(p(j(\tau),(g_{\left[\begin{smallmatrix}0\\1/N\end{smallmatrix}\right]}(\tau)^{12N})^\gamma)
/d(j(\tau)))(\omega)\quad\textrm{because $\gamma$ fixes $j(\tau)\in\mathcal{F}_1$}\\
&=&(p(j(\tau),(g_{{^t}\gamma\left[\begin{smallmatrix}0\\1/N\end{smallmatrix}\right]}(\tau)^{12N}))
/d(j(\tau)))(\omega)\quad\textrm{by Proposition \ref{Siegelfamily}, (F2) and (F3)}\\
&=&(p(j(\tau),g_{\left[\begin{smallmatrix}a/N\\b/N\end{smallmatrix}\right]}(\tau)^{12N})
/d(j(\tau)))(\omega)\\
&=&(p(j(\omega),g_\mathfrak{n}^{12N}(C_0))
/d(j(\omega)).
\end{eqnarray*}
Thus we achieve that
\begin{eqnarray*}
K_\mathfrak{n}&=&H_K(f_\mathfrak{n}(C_0))\quad\textrm{by Theorem \ref{Frickesingular}}\\
&=&H_K((p(j(\omega),g_\mathfrak{n}^{12N}(C_0))
/d(j(\omega)))\\
&\subseteq&H_K(j(\omega),g_\mathfrak{n}^{12N}(C_0))\quad\textrm{since $d(j(\omega))\neq0$}\\
&=&H_K(g_\mathfrak{n}^{12N}(C_0))\quad\textrm{because $H_K=K(j(\omega))$}\\
&\subseteq&K_\mathfrak{n}\quad\textrm{by Proposition \ref{invariant}}.
\end{eqnarray*}
This proves $K_\mathfrak{n}=H_K(g_\mathfrak{n}^{12N}(C_0))$, as desired.
\end{proof}

\begin{remark}
We conjecture that $d_N(\tau_K)\neq0$
for all integers $N\geq2$ and all imaginary quadratic fields $K$
other than $\mathbb{Q}(\sqrt{-1})$ and $\mathbb{Q}(\sqrt{-3})$.
\end{remark}

\begin{corollary}\label{maincor}
Let $h_K$ be the class number of $K$ and
$\ell_N=[\mathcal{F}_1(g_{\left[\begin{smallmatrix}0\\1/N\end{smallmatrix}
\right]}(\tau)^{12N}):\mathcal{F}_1]$.
If $h_K>N\ell_N(\ell_N-1)/2$,
then $g_\mathfrak{n}^{12N}(C)$ generates $K_\mathfrak{n}$ over $H_K$.
\end{corollary}
\begin{proof}
Letting $g_1,g_2,\ldots,g_{\ell_N}$ be all the zeros of the polynomial $\min(g_{\left[\begin{smallmatrix}0\\1/N\end{smallmatrix}\right]}(\tau)^{12N},\mathcal{F}_1)\in
\mathcal{F}_1[X]$ we see that
\begin{eqnarray*}
\mathrm{ord}_q(d_N(\tau))&=&
\mathrm{ord}_q(\prod_{1\leq k_1<k_2\leq\ell_N}(g_{k_1}-g_{k_2})^2)\\
&=&2\sum_{1\leq k_1<k_2\leq\ell_N}\mathrm{ord}_q(g_{k_1}-g_{k_2})\\
&\geq&2\sum_{1\leq k_1<k_2\leq\ell_N}\min\{\mathrm{ord}_q(g_{k_1}),
\mathrm{ord}_q(g_{k_2})\}\\
&\geq&2\sum_{1\leq k_1<k_2\leq\ell_N}6N\mathbf{B}_2(1/2)\\
&&\textrm{by Remark \ref{zeros}, (\ref{OrderSiegel}) and the shape of the graph $Y=\mathbf{B}_2(X)$}\\
&=&-N\ell_N(\ell_N-1)/2.
\end{eqnarray*}
Let $d(X)$ be the polynomial in $\mathbb{Q}[X]$ given in (\ref{d_Nd}).
Since $\mathrm{ord}_q(j(\tau))=-1$, we attain
\begin{equation*}
\deg(d(X))\leq N\ell_N(\ell_N-1)/2.
\end{equation*}
Now, the assumption $h_K=\mathrm{deg}(\min(j(\tau_K),K))>N\ell_N(\ell_N-1)/2$ implies that
\begin{equation*}
d_N(\tau_K)=d(j(\tau_K))
\neq0.
\end{equation*}
Thus the result follows from Theorem \ref{Siegelinvariant}.
\end{proof}

\section {Shimura's reciprocity law}

Let $K$ be an imaginary quadratic field other than $\mathbb{Q}(\sqrt{-1})$ and
$\mathbb{Q}(\sqrt{-3})$. For a positive integer $N$, let $\mathfrak{n}=N\mathcal{O}_K$ and
\begin{equation*}
\mathcal{F}_{N,K}=\{h(\tau)\in\mathcal{F}_N~|~\textrm{$h(\tau)$ is finite at $\tau_K$}\}.
\end{equation*}
As a consequence of the theory of complex multiplication we obtain the following proposition.

\begin{proposition}\label{CM}
We have
$K_\mathfrak{n}=K(h(\tau_K)~|~h(\tau)\in\mathcal{F}_{N,K})$.
\end{proposition}
\begin{proof}
See \cite[Chapter 10, Corollary to Theorem 2]{Lang}.
\end{proof}

\begin{proposition}[Shimura's reciprocity law]\label{Shimura}
Let $\min(\tau_K,\mathbb{Q})=X^2+bX+c\in\mathbb{Z}[X]$.
The matrix group
\begin{equation*}
\mathcal{W}_{N,K}=\bigg\{\left[\begin{matrix}
t-bs & -cs\\s&t
\end{matrix}\right]\in\mathrm{GL}_2(\mathbb{Z}/N\mathbb{Z})~
|~t,s\in\mathbb{Z}/N\mathbb{Z}\bigg\}
\end{equation*}
gives rise to the isomorphism
\begin{eqnarray*}
\mathcal{W}_{N,K}/\{\pm I_2\}&\rightarrow&\mathrm{Gal}(K_\mathfrak{n}/H_K)\\
\alpha&\mapsto&h(\tau_K)\mapsto h(\tau)^\alpha(\tau_K),~h(\tau)\in\mathcal{F}_{N,K}.
\end{eqnarray*}
\end{proposition}
\begin{proof}
See \cite[$\S$3]{Stevenhagen}.
\end{proof}

\begin{remark}\label{correspond}
Let $x=s\tau_K+t\in\mathcal{O}_K$ with $s,t\in\mathbb{Z}$. If $x\mathcal{O}_K$ is relatively prime to $\mathfrak{n}$, then the class $[x\mathcal{O}_K]$ in
$P_K(\mathfrak{n})/P_{K,1}(\mathfrak{n})$ corresponds to the matrix
$\left[\begin{matrix}t-bs&-cs\\s&t\end{matrix}\right]\in\mathcal{W}_{N,K}/\{\pm I_2\}$
(\cite[Chapter 11, $\S$1]{Lang} and \cite{Stevenhagen}).
\end{remark}

\begin{lemma}\label{quotientremark}
Assume that $N\equiv0\Mod{4}$. We have
\begin{equation*}
g_{\left[\begin{smallmatrix}1/2\\1/2+1/N\end{smallmatrix}\right]}(\tau_K)^{12N}/
g_{\left[\begin{smallmatrix}0\\1/N\end{smallmatrix}\right]}(\tau_K)^{12N}=
g_\mathfrak{n}^{12N}(C)/g_\mathfrak{n}^{12N}(C_0),
\end{equation*}
where $C=[((N/2)\tau_K+N/2+1)\mathcal{O}_K]$, $C_0=[\mathcal{O}_K]\in\mathrm{Cl}(\mathfrak{n})$.
This value is a unit in $K_\mathfrak{n}$.
\end{lemma}
\begin{proof}
If we take $\mathfrak{c}=\mathcal{O}_K\in C_0$, then we have
\begin{equation*}
\mathfrak{n}\mathfrak{c}^{-1}=\mathfrak{n}=[N\tau_K,N]
\quad\textrm{and}\quad
1=0(N\tau_K)+(1/N)N.
\end{equation*}
So we get by the definition (\ref{definvariant})
\begin{equation}\label{g_N}
g_\mathfrak{n}^{12N}(C_0)=g_{\left[\begin{smallmatrix}0\\1/N\end{smallmatrix}\right]}(\tau_K)^{12N}. \end{equation}
By Remark \ref{correspond}, the class $C=[((N/2)\tau_K+N/2+1)\mathcal{O}_K]\in P_{K}(\mathfrak{n})/P_{K,1}(\mathfrak{n})$ corresponds to
\begin{equation*}
\alpha=\left\{
\begin{array}{ll}
\left[\begin{matrix}
1 & ((1-d_K)/4)(N/2)\\N/2&N/2+1
\end{matrix}\right] & \textrm{if}~d_K\equiv1\Mod{4},\vspace{0.1cm}\\
\left[\begin{matrix}
N/2+1 & (d_K/4)(N/2)\\N/2&N/2+1
\end{matrix}\right] & \textrm{if}~d_K\equiv0\Mod{4}
\end{array}\right.
\end{equation*}
in $\mathcal{W}_{N,K}/\{\pm I_2\}$.
We then deduce that
\begin{eqnarray*}
g_\mathfrak{n}^{12N}(C)/g_\mathfrak{n}^{12N}(C_0)&=&
g_\mathfrak{n}^{12N}(C_0)^{\sigma_\mathfrak{n}(C)}/g_\mathfrak{n}^{12N}(C_0)
\quad\textrm{by Proposition \ref{invariant}}\\
&=&(g_{\left[\begin{smallmatrix}0\\1/N\end{smallmatrix}\right]}(\tau_K)^{12N})^{\sigma_\mathfrak{n}(C)}
/g_{\left[\begin{smallmatrix}0\\1/N\end{smallmatrix}\right]}(\tau_K)^{12N}
\quad\textrm{by (\ref{g_N})}\\
&=&(g_{\left[\begin{smallmatrix}0\\1/N\end{smallmatrix}\right]}(\tau)^{12N})^\alpha(\tau_K)/
g_{\left[\begin{smallmatrix}0\\1/N\end{smallmatrix}\right]}(\tau_K)^{12N}
\quad\textrm{by Proposition \ref{Shimura}}\\
&=&g_{{^t}\alpha\left[\begin{smallmatrix}0\\1/N\end{smallmatrix}\right]}(\tau)^{12N}(\tau_K)/
g_{\left[\begin{smallmatrix}0\\1/N\end{smallmatrix}\right]}(\tau_K)^{12N}
\quad\textrm{by Proposition \ref{Siegelfamily}, (F2) and (F3)}\\
&=&g_{\left[\begin{smallmatrix}1/2\\1/2+1/N\end{smallmatrix}\right]}(\tau_K)^{12N}/
g_{\left[\begin{smallmatrix}0\\1/N\end{smallmatrix}\right]}(\tau_K)^{12N}.
\end{eqnarray*}
This is a unit in $K_\mathfrak{n}$ by Proposition \ref{invariant}.
\end{proof}

\section {Invariants with small exponents}

Let $K$ be an imaginary quadratic field of discriminant $d_K$. For a positive integer $N$, let  $\mathfrak{n}=N\mathcal{O}_K$
Throughout this section we assume that
\begin{itemize}
\item[(i)] $N\geq4$ and $N\equiv0\Mod{2}$,
\item[(ii)] $|d_K|\geq4N^{4/3}$ ($>25$) and $d_K\equiv0\Mod{4}$,
\end{itemize}

\begin{lemma}\label{inequalitylemma}
Let $\mathbf{v}=\left[\begin{matrix}a/N\\b/N\end{matrix}
\right]\in\mathcal{V}_N$.
\begin{itemize}
\item[\textup{(i)}]
If $\mathbf{v}\not\equiv\pm\left[
\begin{matrix}0\\1/N\end{matrix}\right]\Mod{\mathbb{Z}^2}$, then we have
$|g_{\mathbf{v}}(\tau_K)|>|g_{\left[\begin{smallmatrix}0\\1/N\end{smallmatrix}\right]}(\tau_K)|$.
\item[\textup{(ii)}] We also get
$|g_\mathbf{v}(\tau_K)|\leq
|g_{\left[\begin{smallmatrix}1/2\\1/2+1/N\end{smallmatrix}\right]}(\tau_K)|$.
\end{itemize}
\end{lemma}
\begin{proof}
Since $g_\mathbf{v}(\tau)^{12N}$ depends
only on $\pm\mathbf{v}\Mod{\mathbb{Z}^2}$ by Proposition \ref{Siegelfamily} and (F2), we may assume that
$0\leq a/N\leq1/2$ and $0\leq b/N<1$.
Now that $d_K\equiv0\Mod{4}$, we have $\tau_K=\sqrt{d_K}/2$.
We obtain by (\ref{FourierSiegel}) that
\begin{eqnarray}
|g_\mathbf{v}(\tau_K)|^2&=&A^{\mathbf{B}_2(a/N)}
(1-2\cos(2\pi b/N)A^{a/N}+A^{2a/N})\nonumber\\
&&\times\prod_{n=1}^\infty\{(1-2\cos(2\pi b/N)A^{n+a/N}+A^{2(n+a/N)})\nonumber\\
&&\times(1-2\cos(2\pi b/N)A^{n-a/N}+A^{2(n-a/N)})\},\label{AAA}
\end{eqnarray}
where $A=e^{-\pi\sqrt{|d_K|}}$ ($<e^{-5\pi}$).\\
(i) If $a/N=0$, then
the assumption $\mathbf{v}\not\equiv\pm\left[
\begin{matrix}0\\1/N\end{matrix}\right]\Mod{\mathbb{Z}^2}$ yields $2/N\leq b/N\leq (N-2)/N$. We then obtain by (\ref{AAA}) and the shape of the graph $Y=\cos X$ that
\begin{equation*}
|g_\mathbf{v}(\tau_K)|>|g_{\left[\begin{smallmatrix}0\\1/N\end{smallmatrix}\right]}(\tau_K)|.
\end{equation*}
\par
Now, let $1/N\leq a/N\leq 1/2$. We find by (\ref{AAA}) and the shape of the graph
$Y=\mathbf{B}_2(X)$ that
\begin{equation}\label{ine}
|g_\mathbf{v}(\tau_K)|\geq|g_{\left[\begin{smallmatrix}a/N\\0\end{smallmatrix}\right]}(\tau_K)|.
\end{equation}
Furthermore, we derive by (\ref{AAA}) that
\begin{eqnarray*}
\frac{|g_{\left[\begin{smallmatrix}0\\1/N\end{smallmatrix}\right]}(\tau_K)|}
{|g_{\left[\begin{smallmatrix}a/N\\0\end{smallmatrix}\right]}(\tau_K)|}
&\leq&
\frac{A^{(1/2)\mathbf{B}_2(0)}2\sin(\pi/N)\prod_{n=1}^\infty(1+A^{2n})}
{A^{(1/2)\mathbf{B}_2(a/N)}(1-A^{a/N})\prod_{n=1}^\infty
(1-A^{n+a/N})(1-A^{n-a/N})}\\
&&\textrm{since $\cos(2\pi/N)\geq0$ for $N\geq4$}\\
&\leq&\frac{2\sin(\pi/N)A^{(1/2)\mathbf{B}_2(0)}\prod_{n=1}^\infty(1+A^{n/4})}
{A^{(1/2)\mathbf{B}_2(1/N)}(1-A^{1/N})\prod_{n=1}^\infty
(1-A^{n/2})^2}\\
&&\textrm{by the shape of the graph $Y=\mathbf{B}_2(X)$ and the fact $1/N\leq a/N\leq1/2$}\\
&\leq&\frac{2\sin(\pi/N)A^{(1/2)(\mathbf{B}_2(0)-\mathbf{B}_2(1/N))}}{(1-A^{1/N})}
\prod_{n=1}^\infty(1+A^{n/4})^3\\
&&\textrm{by the inequality $(1-A^{n/2})(1+A^{n/4})>1$ due to $A<e^{-5\pi}$}\\
&\leq&\frac{2\sin(\pi/N)A^{(1/2N)(1-1/N)}}{(1-A^{1/N})}
e^{\sum_{n=1}^\infty 3A^{n/4}}\quad\textrm{by the fact $1+X<e^X$ for $X>0$}\\
&\leq&\frac{2\sin(\pi/N)e^{-\pi N^{-1/3}(1-1/N)}}
{(1-e^{-2\pi N^{-1/3}})}e^{3A^{1/4}/(1-A^{1/4})}\quad
\textrm{by the assumption $|d_K|\geq 4N^{4/3}$}\\
&<&0.4e^{3e^{-5\pi/4}/(1-e^{-5\pi/4})}
\quad\textrm{by considering the graph of $Y=\frac{2\sin(\pi/X)e^{-\pi X^{-1/3}(1-1/X)}}
{(1-e^{-2\pi X^{-1/3}})}$}\\
&&\textrm{at $X\geq4$ and the fact $A<e^{-5\pi}$}\\
&<&1.
\end{eqnarray*}
Thus we attain by (\ref{ine}) that
\begin{equation*}
|g_\mathbf{v}(\tau_K)|>|g_{\left[\begin{smallmatrix}0\\1/N\end{smallmatrix}\right]}(\tau_K)|.
\end{equation*}
(ii) Considering the shape of the graphs $Y=\mathbf{B}_2(X)$ and $Y=\cos X$ and the fact $\mathbf{v}\in\mathcal{V}_N$, we deduce that
\begin{equation*}
|g_\mathbf{v}(\tau_K)|\leq\max\{
|g_{\left[\begin{smallmatrix}1/2-1/N\\1/2\end{smallmatrix}\right]}(\tau_K)|,
|g_{\left[\begin{smallmatrix}1/2\\1/2+1/N\end{smallmatrix}\right]}(\tau_K)|
\}.
\end{equation*}
So it suffices to show
\begin{equation*}
|g_{\left[\begin{smallmatrix}1/2-1/N\\1/2\end{smallmatrix}\right]}(\tau_K)|\leq
|g_{\left[\begin{smallmatrix}1/2\\1/2+1/N\end{smallmatrix}\right]}(\tau_K)|
\end{equation*}
in order to prove $|g_\mathbf{v}(\tau_K)|\leq
|g_{\left[\begin{smallmatrix}1/2\\1/2+1/N\end{smallmatrix}\right]}(\tau_K)|$.
Now, we derive that
\begin{eqnarray*}
\frac{|g_{\left[\begin{smallmatrix}1/2-1/N\\1/2\end{smallmatrix}\right]}(\tau_K)|}
{|g_{\left[\begin{smallmatrix}1/2\\1/2+1/N\end{smallmatrix}\right]}(\tau_K)|}
&\leq&
\frac{A^{(1/2)\mathbf{B}_2(1/2-1/N)}(1+A^{1/2-1/N})\prod_{n=1}^\infty(1+A^{n+1/2-1/N})(1+A^{n-1/2+1/N})}
{A^{(1/2)\mathbf{B}_2(1/2)}}\\
&&\textrm{by (\ref{AAA}) and the fact $\cos(2\pi/N)\geq0$ for $N\geq4$}\\
&\leq&A^{1/2N^2}(1+A^{1/4})\prod_{n=1}^\infty(1+A^{n+1/4})(1+A^{n-1/2})\quad\textrm{because $A<1$ and $N\geq4$}\\
&\leq&A^{1/2N^2}\prod_{n=1}^\infty(1+A^{n/4})\\
&\leq&A^{1/2N^2}e^{\sum_{n=1}^\infty A^{n/4}}\quad\textrm{due to the fact $1+X<e^X$ for all $X>0$}\\
&\leq&e^{-\pi/N^{4/3}}e^{\sum_{n=1}^\infty e^{-(\pi N^{2/3}/2)n}}
\quad\textrm{since $|d_K|\geq4N^{4/3}$}\\
&=&e^{-\pi/N^{4/3}+e^{-\pi N^{2/3}/2}/(1-e^{-\pi N^{2/3}/2})}\\
&<&e^{-\pi/N^{4/3}+8/\pi^2N^{4/3}(1-e^{-\pi N^{2/3}/2})}\quad
\textrm{because $e^{-X}<2X^{-2}$ for all $X>0$}\\
&=&e^{(\pi/N^{4/3})(-1+8/\pi^3(1-e^{-\pi N^{2/3}/2}))}\\
&\leq&e^{(\pi/N^{4/3})(-1+8/\pi^3(1-e^{-\pi2^{1/3}}))}\quad\textrm{owing to the fact $N\geq4$}\\
&<&1.
\end{eqnarray*}
This proves (ii).
\end{proof}

\begin{theorem}\label{smallinvariant}
The special value
\begin{equation}\label{quotientinv}
\zeta_{2N}^{-4/\gcd(4,N)}
(g_{\left[\begin{smallmatrix}1/2\\1/2+1/N\end{smallmatrix}\right]}(\tau_K)/
g_{\left[\begin{smallmatrix}0\\1/N\end{smallmatrix}\right]}(\tau_K))^{8/\gcd(4,N)}
\end{equation}
generates $K_\mathfrak{n}$ over $H_K$.
Moreover, if $N\equiv0\Mod{4}$, then it is a $6N$-th root of the unit
$g_\mathfrak{n}^{12N}(C)/g_\mathfrak{n}^{12N}(C_0)$, where
$C=[((N/2)\tau_K+N/2+1)\mathcal{O}_K]$ and $C_0=[\mathcal{O}_K]$.
\end{theorem}
\begin{proof}
Since
$\zeta_{2N}^{-4/\gcd(4,N)}(g_{\left[\begin{smallmatrix}1/2\\1/2+1/N\end{smallmatrix}\right]}(\tau)/
g_{\left[\begin{smallmatrix}0\\1/N\end{smallmatrix}\right]}(\tau))^{8/\gcd(4,N)}$
belongs to $\mathcal{F}_N$ by Lemma \ref{modularity}, its special value at $\tau_K$ lies in $K_\mathfrak{n}$ by Proposition \ref{CM}.
Let $\sigma\in\mathrm{Gal}(K_\mathfrak{n}/H_K)$ such that $\sigma\neq\mathrm{id}$.
We observe that
\begin{eqnarray*}
\bigg|\bigg(\frac{g_{\left[\begin{smallmatrix}1/2\\1/2+1/N\end{smallmatrix}\right]}(\tau_K)^{12N}}
{g_{\left[\begin{smallmatrix}0\\1/N\end{smallmatrix}\right]}(\tau_K)^{12N}}
\bigg)^\sigma\bigg|&=&
\bigg|\frac{(g_{\left[\begin{smallmatrix}1/2\\1/2+1/N\end{smallmatrix}\right]}(\tau_K)^{12N})^\sigma}
{(g_{\left[\begin{smallmatrix}0\\1/N\end{smallmatrix}\right]}(\tau_K)^{12N})^\sigma}\bigg|
\quad\textrm{by Lemma \ref{modularity} and Proposition \ref{CM}}\\
&=&\bigg|\frac{(g_{\left[\begin{smallmatrix}1/2\\1/2+1/N\end{smallmatrix}\right]}(\tau)^{12N})^\alpha
(\tau_K)}
{(g_{\left[\begin{smallmatrix}0\\1/N\end{smallmatrix}\right]}(\tau)^{12N})^\alpha(\tau_K)}\bigg|
\\&&\textrm{for some $\alpha\in\mathcal{W}_{N,K}/\{\pm I_2\}$ by Proposition \ref{Shimura}}\\
&=&\bigg|\frac{g_\mathbf{u}(\tau_K)^{12N}}
{g_\mathbf{v}(\tau_K)^{12N}}\bigg|\\
&&\textrm{for some $\mathbf{u},\mathbf{v}\in\mathcal{V}_N$ by Proposition \ref{Siegelfamily}, (F2) and (F3)}\\
&<&\bigg|\frac{g_{\left[\begin{smallmatrix}1/2\\1/2+1/N\end{smallmatrix}\right]}(\tau_K)^{12N}}
{g_{\left[\begin{smallmatrix}0\\1/N\end{smallmatrix}\right]}(\tau_K)^{12N}}\bigg|\quad\textrm{by Lemma \ref{inequalitylemma}}.
\end{eqnarray*}
This implies that the value in (\ref{quotientinv}) generates $K_\mathfrak{n}$ over $H_K$.
The second part of the theorem follows from Lemma \ref{quotientremark}.
\end{proof}

\begin{remark}\label{lastremark}
\begin{itemize}
\item[(i)] Suppose that $N$ is not a power of $2$, and so
$N$ has at least two prime factors. Then, both $g_\mathbf{v}(\tau)^{12N}$ and $g_\mathbf{v}(\tau)^{-12N}$ are integral over $\mathbb{Z}[j(\tau)]$
for any $\mathbf{v}\in\mathcal{V}_N$ (\cite[Chapter 2, Theorem 2.2]{K-L}).
Moreover, since $j(\tau_K)$ is an algebraic integer (\cite[Theorem 4.14]{Shimura}),
we see that $g_\mathbf{v}(\tau_K)^{12N}$ is a unit. Therefore, the invariant in (\ref{quotientinv}) is a unit.
\item[(ii)] For any $\mathbf{u},\mathbf{v}\in\mathbb{Q}^2\setminus\mathbb{Z}^2$ such that $\mathbf{u}\not
\equiv\pm\mathbf{v}\Mod{\mathbb{Z}^2}$ we have the relation
\begin{equation*}
\wp_\mathbf{u}(\tau)-\wp_\mathbf{v}(\tau)
=-(g_{\mathbf{u}+\mathbf{v}}(\tau)
g_{\mathbf{u}-\mathbf{v}}(\tau)/
g_\mathbf{u}(\tau)^2g_\mathbf{v}(\tau)^2)\eta(\tau)^4
\end{equation*}
(\cite[p.51]{K-L}).
This relation together with (\ref{Fricke}) and (\ref{FourierSiegel}) yields
\begin{equation}\label{Frickequotient}
\frac{f_{\left[\begin{smallmatrix}0\\1/N\end{smallmatrix}\right]}(\tau_K)
-f_{\left[\begin{smallmatrix}1/2\\1/2\end{smallmatrix}\right]}(\tau_K)}
{f_{\left[\begin{smallmatrix}0\\1/2\end{smallmatrix}\right]}(\tau_K)-
f_{\left[\begin{smallmatrix}1/2\\1/2\end{smallmatrix}\right]}(\tau_K)}
=-\zeta_{2N}
\frac{g_{\left[\begin{smallmatrix}1/2\\1/2+1/N\end{smallmatrix}\right]}(\tau_K)^2}
{g_{\left[\begin{smallmatrix}0\\1/N\end{smallmatrix}\right]}(\tau_K)^2}\cdot
\zeta_4^3\frac{g_{\left[\begin{smallmatrix}0\\1/2\end{smallmatrix}\right]}(\tau_K)^2}
{g_{\left[\begin{smallmatrix}1/2\\0\end{smallmatrix}\right]}(\tau_K)^2}.
\end{equation}
Since $f_\mathfrak{n}(C_0)=f_{\left[\begin{smallmatrix}
0\\1/N\end{smallmatrix}\right]}(\tau_K)$
generates $K_\mathfrak{n}$ over $H_K$ by Theorem \ref{Frickesingular}, and
$f_{\left[\begin{smallmatrix}
1/2\\1/2\end{smallmatrix}\right]}(\tau_K)$
and
$f_{\left[\begin{smallmatrix}
0\\1/2\end{smallmatrix}\right]}(\tau_K)$
lie in $K_{2\mathcal{O}_K}$ by Propositions \ref{Frickefamily} and \ref{CM},
the value in the left side of (\ref{Frickequotient}) generates $K_\mathfrak{n}$ over $K_{2\mathcal{O}_K}$.
Now, assume that $N\equiv0\Mod{4}$.
Since $\zeta_4^3g_{\left[\begin{smallmatrix}0\\1/2\end{smallmatrix}\right]}(\tau_K)^2/
g_{\left[\begin{smallmatrix}1/2\\0\end{smallmatrix}\right]}(\tau_K)^2$
belongs to $K_{4\mathcal{O}_K}$ by Lemma \ref{modularity} and Proposition \ref{CM},
the value
\begin{equation*}
\zeta_{2N}g_{\left[\begin{smallmatrix}1/2\\1/2+1/N\end{smallmatrix}\right]}(\tau_K)^2
/g_{\left[\begin{smallmatrix}0\\1/N\end{smallmatrix}\right]}(\tau_K)^2
\end{equation*}
generates $K_{\mathfrak{n}}$ over $K_{4\mathcal{O}_K}$.
\end{itemize}
\end{remark}

\begin{example}\label{lastexample}
Let $K=\mathbb{Q}(\sqrt{-10})$ and $\mathfrak{n}=4\mathcal{O}_K$. Consider the special value
\begin{equation*}
x=\zeta_8^7g_{\left[\begin{smallmatrix}1/2\\3/4\end{smallmatrix}\right]}(\sqrt{-10})^2/
g_{\left[\begin{smallmatrix}0\\1/4\end{smallmatrix}\right]}(\sqrt{-10})^2.
\end{equation*}
This value generates $K_\mathfrak{n}$ over $H_K$ as an algebraic unit by Theorem \ref{smallinvariant}.
Furthermore, since $x$ is a real number by (\ref{FourierSiegel}), we see that
\begin{equation*}
[K(x):K]=[K(x):\mathbb{Q}]/[K:\mathbb{Q}]
=[K(x):\mathbb{Q}(x)][\mathbb{Q}(x):\mathbb{Q}]/[K:\mathbb{Q}]
=[\mathbb{Q}(x):\mathbb{Q}].
\end{equation*}
Hence the minimal polynomial of $x$ over $K$ has integer coefficients.
By using Proposition \ref{Shimura}, \cite{Stevenhagen} and \cite[Chapter 2, $\S$1]{K-L} one can readily find
all the Galois conjugates of $x$ over $K$ (possibly with some multiplicity) as follows:
\begin{equation*}
\begin{array}{ll}
x_1=\zeta_8^7g_{\left[\begin{smallmatrix}1/2\\3/4\end{smallmatrix}\right]}(\sqrt{-10})^2/
g_{\left[\begin{smallmatrix}0\\1/4\end{smallmatrix}\right]}(\sqrt{-10})^2, &
x_2=\zeta_8^5g_{\left[\begin{smallmatrix}1/4\\3/4\end{smallmatrix}\right]}(\sqrt{-10})^2/
g_{\left[\begin{smallmatrix}1/4\\1/4\end{smallmatrix}\right]}(\sqrt{-10})^2,\\
x_3=\zeta_8^3g_{\left[\begin{smallmatrix}0\\3/4\end{smallmatrix}\right]}(\sqrt{-10})^2/
g_{\left[\begin{smallmatrix}1/2\\1/4\end{smallmatrix}\right]}(\sqrt{-10})^2,
& x_4=\zeta_8^7g_{\left[\begin{smallmatrix}3/4\\3/4\end{smallmatrix}\right]}(\sqrt{-10})^2/
g_{\left[\begin{smallmatrix}3/4\\1/4\end{smallmatrix}\right]}(\sqrt{-10})^2,\\
x_5=\zeta_8g_{\left[\begin{smallmatrix}3/4\\1/2\end{smallmatrix}\right]}(\sqrt{-10}/2)^2/
g_{\left[\begin{smallmatrix}1/4\\0\end{smallmatrix}\right]}(\sqrt{-10}/2)^2, &
x_6=\zeta_8g_{\left[\begin{smallmatrix}3/4\\3/4\end{smallmatrix}\right]}(\sqrt{-10}/2)^2/
g_{\left[\begin{smallmatrix}1/4\\3/4\end{smallmatrix}\right]}(\sqrt{-10}/2)^2,\\
x_7=\zeta_8^5g_{\left[\begin{smallmatrix}3/4\\0\end{smallmatrix}\right]}(\sqrt{-10}/2)^2/
g_{\left[\begin{smallmatrix}1/4\\1/2\end{smallmatrix}\right]}(\sqrt{-10}/2)^2,
&
x_8=\zeta_8^3g_{\left[\begin{smallmatrix}3/4\\1/4\end{smallmatrix}\right]}(\sqrt{-10}/2)^2/
g_{\left[\begin{smallmatrix}1/4\\1/4\end{smallmatrix}\right]}(\sqrt{-10}/2)^2.
\end{array}
\end{equation*}
And, one can also compute (by using MAPLE Ver.16) $\min(x,K)$ as
\begin{equation*}
\prod_{k=1}^8(X-x_k)=X^8-72X^7+12X^6+72X^5+38X^4+72X^3+12X^2-72X+1,
\end{equation*}
which is irreducible over $\mathbb{Q}$. Therefore, $x$ generates
$K_\mathfrak{n}$ even over $K$ as a unit. Here, we observe that the
coefficients of $\min(x,K)$ are much smaller than those of
\begin{eqnarray*}
&&\min(g_{\left[\begin{smallmatrix}0\\1/4\end{smallmatrix}\right]}(\sqrt{-10})^{48},K)\\&=&
X^8-181195540256817728X^7
-5775663114562606906112X^6\\
&&-27035464691637377457360896X^5
+541339076030741096821545656320X^4\\
&& -124937615343087944795342556102656X^3
+15661918473435227713231818559848448X^2\\
&&-32831816404527400323644148540243968X+16777216.
\end{eqnarray*}
\end{example}

\bibliographystyle{amsplain}

\address{
National Institute for Mathematical
Sciences\\
Daejeon 305-811\\
Republic of Korea} {hoyunjung@nims.re.kr}
\address{
Department of Mathematical Sciences \\
KAIST \\
Daejeon 305-701 \\
Republic of Korea} {jkkoo@math.kaist.ac.kr}
\address{
Department of Mathematics\\
Hankuk University of Foreign Studies\\
Yongin-si, Gyeonggi-do 449-791\\
Republic of Korea } {dhshin@hufs.ac.kr}

\end{document}